\documentclass[12pt,reqno,a4paper]{amsart}

\usepackage{mdframed}
\usepackage{setspace}
\usepackage{mathrsfs}  
\usepackage{comment}

\usepackage[latin1]{inputenc}

\usepackage[linkcolor=red,colorlinks=true]{hyperref}

\usepackage{bold-extra}
\usepackage{etoolbox}
\usepackage{amsmath}
\usepackage{enumerate}
\usepackage{amssymb}
\usepackage{amscd}
\usepackage{amsthm}
\usepackage{amsfonts}
\usepackage{graphicx}
\usepackage[all,cmtip]{xy}

\patchcmd{\subsection}{-.5em}{.5em}{}{}
\patchcmd{\subsubsection}{-.5em}{.5em}{}{}

\makeatother
\usepackage{hyperref}

\numberwithin{equation}{section}

\newcommand{\SL}{\operatorname{SL}}

\newcommand{\cA}{\mathcal{A}}

\newcommand{\cE}{\mathcal{E}}
\newcommand{\cF}{\mathcal{F}}

\newcommand{\cH}{\mathcal{H}}

\newcommand{\cP}{\mathcal{P}}

\newcommand{\bC}{\mathbb{C}}

\newcommand{\bE}{\mathbb{E}}

\newcommand{\bN}{\mathbb{N}}

\newcommand{\bR}{\mathbb{R}}

\newcommand{\bT}{\mathbb{T}}

\newcommand{\bZ}{\mathbb{Z}}

\newcommand{\ra}{\rightarrow}

\newcommand{\qand}{\quad \textrm{and} \quad}

\def\acts{\curvearrowright}
\newcommand\subsetsim{\mathrel{%
\ooalign{\raise0.2ex\hbox{$\subset$}\cr\hidewidth\raise-0.8ex\hbox{\scalebox{0.9}{$\sim$}}\hidewidth\cr}}}
\newcommand{\eps}{\varepsilon}

\DeclareMathOperator{\Hom}{Hom}

\DeclareMathOperator{\linspan}{span}

\DeclareMathOperator{\Prob}{Prob}

\DeclareMathOperator{\rat}{Rat}
\DeclareMathOperator{\Vol}{Vol}

\renewcommand{\phi}{\varphi}

\definecolor{lichtgrijs}{gray}{0.95}
\mdfdefinestyle{mystyle}{ %
    backgroundcolor=lichtgrijs, %
    linewidth=0pt, %
    innertopmargin=10pt, %
    innerbottommargin=10pt,%
    nobreak=true
}

\theoremstyle{theorem}
\newtheorem{theorem}{Theorem}[section]
\newtheorem{corollary}[theorem]{Corollary}
\newtheorem{proposition}[theorem]{Proposition}
\newtheorem{lemma}[theorem]{Lemma}

\theoremstyle{task}

\theoremstyle{definition}
\newtheorem{definition}[theorem]{Definition}
\newtheorem{remark}[theorem]{Remark}

\newtheorem{example}[theorem]{Example}

\usepackage{changepage}
\usepackage{mathtools}
\usepackage{graphicx}
\usepackage{calc}

\usepackage[toc,page]{appendix}

\usepackage{scalerel}

\usepackage{extarrows}

\DeclareMathOperator{\VS}{VolSpec}

\DeclareMathSymbol{\shortminus}{\mathbin}{AMSa}{"39}
\usepackage{mathtools}
\begin{document}
\bibliographystyle{plain} 
 
\title[Directional expansion in ergodic actions]{Simplices in large sets and directional expansion in ergodic actions}

\author{Michael Bj\"orklund}
\address{Department of Mathematics, Chalmers, Gothenburg, Sweden}
\email{micbjo@chalmers.se}
\thanks{}

\author{Alexander Fish}
\address{School of Mathematics and Statistics F07, University of Sydney, NSW 2006,
Australia}
\curraddr{}
\email{alexander.fish@sydney.edu.au}
\thanks{}

\subjclass[2020]{Primary: 37A30, 37A44 . Secondary: 11B30}
\keywords{Directional ergodicity, expansion, volume of simplices}

\begin{abstract} 
In this paper we study ergodic $\bZ^r$-actions and investigate expansion properties along cyclic subgroups. We show that under some spectral conditions there are always
directions which expand significantly a given measurable set with positive measure. 
Among other things, we use this result to prove that the set of volumes of all $r$-simplices with vertices in a set with positive upper density must contain an infinite arithmetic progression, thus showing a discrete density analogue of a classical result by Graham.
\end{abstract}
\maketitle

\section{Introduction}

The \emph{upper density} $\overline{d}(E)$ of a set $E \subset \bZ^r$ is defined by
\[
\overline{d}(E) = \limsup_{N \ra \infty} \frac{|E \cap [-N,N]^r]}{(2N+1)^r},
\]
and we say that $E$ is \emph{large} if $\overline{d}(E) > 0$. Let $\mathscr{F}_r$ represent all finite subsets of $\bZ^r$, and let $\mathscr{M}_r$ denote  
all maps $\varphi : \bR^r \ra \bR^r$. We think of $\bZ^r$ as the standard unimodular lattice in $\bR^r$. Let $\mathscr{A} \subset \mathscr{F}_r$ and $\mathscr{B} \subset \mathscr{M}_r$ be two sets. The pair $(\mathscr{A},\mathscr{B})$ is called a \emph{density pattern matching} if, for every large set $E \subset \bZ^r$ and $F \in \mathscr{A}$, the set
\[
\mathscr{B}_{F}(E) := \{ \varphi \in \mathscr{B} \, : \, \varphi(F) \subset E\}
\]
is non-empty. In simpler terms, it means that every finite pattern in $\mathscr{A}$ can be transformed into any large set by a function in $\mathscr{B}$. Two fundamental questions in Density Ramsey Theory arise:
\begin{itemize}
\item[$\bullet$] Which pairs $(\mathscr{A},\mathscr{B})$ constitute density pattern 
matchings?
\item[$\bullet$] If $(\mathscr{A},\mathscr{B})$ is a density pattern matching and $E \subset \bZ^r$ is a large set, 
how "large" is the set $\mathscr{B}_{F}(E)$ for a given $F \in \mathscr{A}$? 
\end{itemize}
Typically the elements in $\mathscr{B}$ depend on some parameters, so when we refer to "largeness" of $\mathscr{B}_F(E)$ it will be with respect to these parameters. \\

An archetypical result in this context is the theorem of Furstenberg and Katznelson \cite{FK}, which extends an earlier breakthrough of Szemeredi \cite{Sz}, and asserts that
\[
\mathscr{A} = \mathscr{F}_r \qand \mathscr{B} = \{ \varphi_{a,b}(v) = av + b \, : \, a \in \bZ \setminus \{0\}, \enskip b \in \bZ^r\}
\]
forms a pattern density matching. In fact, they show that for every large set $E \subset \bZ^r$ and finite subset $F \subset \bZ^r$, there exist $A \subset \bZ \setminus {0}$ and $B \subset \bZ^r$ with positive densities such that
\[
\mathscr{B}_F(E) \supset \{ \varphi_{a,b} \, : \, a \in A, \enskip b \in B \}.
\]
Numerous extensions of this fundamental result have been explored over the years. \\

More recently, Magyar \cite{M} has introduced a captivating class of density pattern matchings. To illustrate these examples, consider $p \geq 2$ and assume $r > 2p + 4$. Let $\cA$ denote the set of all subsets comprising $p+1$ affinely independent vectors in $\bZ^r$. For $m = (m_1,m_2) \in \bN^2$, a rotation $u \in O(r)$, and $b \in \bR^r$, define
\[
\varphi_{m,u,b}(v) = m_1 \sqrt{m_2} u(v) + b, \quad v \in \bR^r,
\]
Now, consider the set $\mathscr{B}$, which consists of all maps $\varphi_{m,u,b}$ with parameters $(m,u,b)$ as described above. Magyar's theorem \cite[Theorem 1.1]{M} establishes that $(\mathscr{A},\mathscr{B})$ forms a density pattern matching, and for every large set $E$ and $F \in \mathscr{A}$, there are an integer $m_1$ (depending solely on the density of $E$) and an integer $n$ such that for all 
integers $m_2 \geq n$, there exist a rotation $u \in O(r)$ and a vector $b \in \bR^r$ such that $\varphi_{m,u,b}(F) \subset E$. Several generalizations of this result have been proven, see for instance \cite{LM}. It is worth noting that these results typically demand a substantial dimensionality, with $r$ being significantly larger than the size of the patterns one aims to map into large sets. For instance, in the aforementioned result, at least 9 dimensions are required to map any three affinely independent vectors into any large set. A rich body of literature is dedicated to analogous embedding problems in Euclidean spaces, exemplified by \cite{Bo,LM2,Z}.

A different exploration of this theme was undertaken by the authors and Bulinski in a series of papers \cite{BB,BF1,F}. Here, we briefly highlight a key combinatorial result in \cite{BB}. Let $\Gamma$ be a "sufficiently large" subgroup of $\SL_r(\bZ)$. For every large set $E \subset \bZ^r$ and $p \geq 1$, there is an integer $n$ such that, for every finite set $F = \{v_o,\ldots,v_p\} \subset n \cdot \bZ^r$, there are $\gamma_o,\ldots,\gamma_p \in \Gamma$ and $b \in E$ such that
\begin{equation}
\label{bfkincl}
\{\gamma_o(v_o) + b,\ldots,\gamma_p (v_p) + b\} \subset E.
\end{equation}
This result differs from the theorems of Furstenberg and Katznelson and Magyar in several aspects. Notably, it does not involve mapping elements in the finite set $F$ into $E$ by the same function. Additionally, only finite subsets of $n \cdot \bZ^r$ can be mapped into the set $E$, where the integer $n$ only depends on $|F|$ and the set $E$. This constraint is crucial; for instance, if $E = n_o \cdot \bZ^r$ for some $n_o \geq 2$ and $v \in \bZ^r$ has relatively prime coordinates, then the set $F = \{0,v\}$ cannot be mapped into $E$ as described in \eqref{bfkincl}. On the other hand, the dilation $n$ only depends on $E$ and not the finite set $F$, as in the theorem of Furstenberg and Katznelson. 

\subsection{Main combinatorial result}

Our first theorem in this paper can be seen as an amalgamation of the last two results mentioned above. Recall that a vector $\lambda = (k_1,\ldots,k_r) \in \bZ^r$ is \emph{primitive} if $\gcd(k_1,\ldots,k_r) = 1$.

\begin{theorem}
\label{ThmIntoComb}
For every large set $E \subset \bZ^r$ and $p \geq 2$, there are positive integers $n$
and $m_1$ and a primitive vector $\lambda \in \bZ^r$ such that for all $\lambda_2,\ldots,\lambda_p \in \bZ^r$, there are $m_2,\ldots,m_p \in \bZ \setminus \{0\}$ and $\lambda_o \in E$ such that
\[
\lambda_o + m_1 n \lambda \in E, \enskip \lambda_o + m_2 n \lambda + n \lambda_2 \in E, \ldots \lambda_o + m_p n \lambda + n \lambda_p \in E.
\]
\end{theorem}

It is essential to highlight that there is no dimension constraint concerning $p$. Furthermore, akin to the outcomes in \cite{BB}, we do not map the finite set ${\lambda_2,\ldots,\lambda_p}$ into $E$ using the same affine function. While the dilation is always the same for all elements, the translation component varies from element to element. \\

Our main application of Theorem \ref{ThmIntoComb} is concerned with volume spectra of large sets. To elucidate this concept, consider $r+1$ points $v_o,v_1,\ldots,v_r \in \bZ^r$ such that the differences ${v_1-v_o,\ldots,v_r-v_o}$ are linearly independent, and form the $r$-simplex $S(v_o,\ldots,v_r) \subset \bR^r$, defined by
\[
S(v_o,\ldots,v_r) = \Big\{ \sum_{k=0}^r p_k v_k \, : \, (p_o,\ldots,p_r) \in [0,1]^{r+1}, \enskip \sum_{k=0}^r p_k = 1 \Big\}.
\]
The elements ${v_o,\ldots,v_r}$ are referred to as the \emph{vertices} of $S(v_o,\ldots,v_r)$. For a large set $E \subset \bZ^r$, our focus now lies in understanding the structure of the set consisting of the volumes of all $r$-simplices with vertices in $E$. This set is termed the \emph{volume spectrum} of $E$, and will be denoted by $\VS_r(E)$. It is well-known (see e.g., \cite{S}) that
\[
\Vol_r(S(v_o,\ldots,v_r)) = \frac{\det(v_1-v_o,v_2-v_o,\ldots,v_r-v_o)}{r!},
\]
where $\Vol_r$ is the (signed) Euclidean volume. Notably, if $(v_1-v_o,\ldots,v_r-v_o)$ are not linearly independent, then $\Vol_r(S(v_o,\ldots,v_r)) = 0$. With this background, we can now formulate and swiftly prove the following corollary of Theorem \ref{ThmIntoComb}. Hopefully the proof will demonstrate the relevance of the patterns guaranteed by this theorem.

\begin{corollary}
For every large set $E \subset \bZ^r$, there exists a non-zero integer $n$ such that
\[
n \cdot (\bZ \setminus \{0\}) \subset r! \cdot \VS_r(E). 
\]
\end{corollary}

\begin{remark}
In the paper \cite{G}, Graham proves a similar result for finite colorings of the Euclidean space $\bR^r$. Specifically, he establishes that for any finite coloring of $\bR^r$ and $\alpha > 0$, there is a monochromatic set $E$ for which there is a $r$-simplex $S$ with vertices in $E$ satisfying $\Vol_r(S) = \alpha$.
\end{remark}

\begin{proof}
Let $E \subset \bZ^r$ be a large set. By Theorem \ref{ThmIntoComb}, there 
are integers $n_o$ and $m_1$ and a primitive vector $\lambda \in \bZ^r$ such that for all $\lambda_2,\ldots,\lambda_r \in \bZ^r$, there
are $m_2,\ldots,m_r \in \bZ \setminus \{0\}$ and $\lambda_o \in E$ such that the 
vectors
\[
v_o = \lambda_o, \enskip v_1 = \lambda_o + m_1 n_o \lambda, \enskip v_k = \lambda_o + m_k n_o \lambda + n_o \lambda_k, \quad k = 2,\ldots,r,
\]
all belong to $E$. Utilizing the multilinearity and alternation of the determinant, we have
\[
\det(v_1-v_o,\ldots,v_r-v_o) = m_1 n_o^r \cdot \det(\lambda,\lambda_2,\ldots,\lambda_r).
\]
Let $n = m_1 n_o^r$. As $\lambda_2,\ldots,\lambda_r$ are arbitrary in $\bZ^{r}$,
we deduce that
\[
r! \cdot \VS_r(E) \supset n \cdot \{ \det(\lambda,\lambda_2,\ldots,\lambda_r) \, : \, \lambda_2,\ldots,\lambda_r \in \bZ^{r} \}.
\]
It is well-known (see for instance \cite[Section II, Chapter 5]{N}) that for every primitive vector $\lambda \in \bZ^r$, there exist $\lambda_2',\ldots,\lambda'_r \in \bZ^r$ such
that $\det(\lambda,\lambda'_2,\ldots,\lambda_r') = 1$. In particular, considering $(r-1)$-tuples $(\lambda_2,\ldots,\lambda_r)$ of the form $(m\lambda_2',\lambda_3',\ldots,\lambda'_r)$, for $m \in \bZ \setminus \{0\}$, we see that
\[
\{ \det(\lambda,\lambda_2,\ldots,\lambda_r) \, : \, (\lambda_2,\ldots,\lambda_r) \in \bZ^{r-1} \} = \bZ \setminus \{0\},
\]
for every primitive $\lambda \in \bZ^r$, and we are done.
\end{proof}

\subsection{Main dynamical results}

It is straightforward to observe, and the details are provided in Section \ref{sec:PrfComb}, that Theorem \ref{ThmIntoComb} can be derived, through Furstenberg's Correspondence Principle, from the following dynamical result. For the remainder of this section, let $(X,\mathscr{B}_X)$ denote a standard Borel space with a measurable 
$\bZ^r$-action
\[
\bZ^r \times X \ra X, \enskip (\lambda,x) \mapsto \lambda.x.
\]
We assume $\mu$ is a $\bZ^r$-invariant and $\bZ^r$-ergodic probability measure on $X$, referring to the pair $(X,\mu)$ as an \emph{ergodic $\bZ^r$-space}. We also fix a $\mu$-measurable set $B \subset X$ with positive $\mu$-measure.

\begin{theorem}
\label{ThmIntroIntersections}
For every $p \geq 2$, there are positive integers $n$ and $m_1$ and a primitive vector $\lambda \in \bZ^r$ with the property that for all $\lambda_2,\ldots,\lambda_p \in \Lambda$, there are positive integers 
$m_2,\ldots,m_p$ such that
\[
\mu\Big( B \cap m_1n\lambda.B \cap \Big(\bigcap_{k=2}^p (m_kn \lambda + n \lambda_k).B \Big)\Big) > 0.
\]
\end{theorem}

To establish Theorem \ref{ThmIntroIntersections}, it is crucial to comprehend the actions of cyclic subgroups of $\bZ^r$ on $(X,\mu)$, even if these actions are not necessarily ergodic. Despite the potential absence of ergodicity for any single cyclic subgroup, we demonstrate that, given certain spectral constraints on the set $B$, there must be a direction that substantially expands $B$. To clarify this concept, we introduce the following definitions. 

\begin{definition}
Given an ergodic $\bZ^r$-space $(X,\mu)$, an element $\lambda \in \bZ^r \setminus {0}$ is termed an \emph{ergodic direction} if the action of the cyclic subgroup $\bZ \lambda$ on $(X,\mu)$ is ergodic. A $\mu$-measurable set $B \subset X$ with positive $\mu$-measure is called \emph{directionally expandable} if, for every $\epsilon > 0$, there exists $\lambda_\epsilon \in \bZ^r \setminus {0}$ such that $\mu(\bZ\lambda_\epsilon.B) > 1 - \epsilon$.
\end{definition}

\noindent These definitions raise two immediate questions:
\vspace{0.1cm}
\begin{itemize}
\item[$\bullet$] Do ergodic directions always exist? \vspace{0.1cm}
\item[$\bullet$] Is every $\mu$-measurable set with positive $\mu$-measure directionally expandable?
\end{itemize}
\vspace{0.1cm}
\noindent As the following examples illustrate, the answer to both questions is no
\begin{example}[No ergodic directions]
This example falls into a class of weakly mixing $\bZ^r$-actions described in \cite[Examples 2.11, 5.11]{RRS}, attributed to Bergelson and Ward. Consider a weakly mixing $\bZ$-action $T \acts (Y,\nu)$ and define the $\bZ^r$-space $(X,\mu)$ by
\[
(X,\mu) = \Big(\prod_{i=1}^\infty Y, \nu^{\bN}\Big) \qand (\lambda.x)_i = T^{\langle \lambda,\eta_i \rangle} x_i, \quad \textrm{for $\lambda \in \bZ^r$ and $i \in \bN$},
\]
where $(\eta_i)$ of $\bZ^r \setminus \{0\}$ is a fixed enumeration of $\bZ^r \setminus \{0\}$ and $\langle \cdot, \cdot \rangle$ denotes the standard inner product on $\bZ^r$. One can readily check that the action $\bZ^r \acts (X,\mu)$ is weakly mixing. Fix a $\nu$-measurable set $B_o \subset Y$ with $0 < \nu(B_o) < 1$. For a given $\lambda \in \Lambda$, we fix an index $i$ such that $\langle \lambda,\eta_i \rangle = 0$, and define
\[
B_i = \{ x \in X \, : \, x_{i} \in B_o\} \subset X.
\]
Then $\mu(B_i) = \nu(B_o) \in (0,1)$ and $B_i$ is invariant under the subgroup $\bZ \lambda < \bZ^r$. Indeed, since $\langle \lambda,\eta_i \rangle = 0$, we have
\[
\lambda.B_i = \{ \lambda.x \in X \, : \, x_i \in B_o \} = 
\{ \lambda.x \in X \, : \, (\lambda.x)_i \in B_o \} = B_i.
\]
In particular, every cyclic subgroup $\bZ\lambda < \bZ^r$ acts non-ergodically on $(X,\mu)$. However, for a given index $i$, note that if $\lambda_o \in \bZ^r$ is instead chosen so that $\langle \lambda_o,\eta_i \rangle = n \neq 0$, then 
\[
\bZ\lambda_o.B_i = \Big\{ x \in X \, : \, x_i \in \bigcup_{k \in \bZ} T^{-kn}B_o\Big\},
\]
which is a $\mu$-conull subset of $X$, since $T$ is weakly mixing on $(Y,\nu)$. Hence, for a fixed index $i$, the set $B_i$ is directionally expandable (in fact, we can take the same direction 
$\lambda_\eps$ for every $\eps > 0$).
\end{example}

\begin{example}[A set which is not directionally expandable]
\label{Ex2}
Consider a finite-index subgroup $\Lambda_o < \bZ^r$ such that $\bZ^r/\Lambda_o$ is not cyclic. Let $X = \bZ^r/\Lambda_o$ with the canonical translation action by $\bZ^r$ and equip $X$ with the normalized counting measure $\mu$. The singleton set $B = \{\Lambda_o\}$ has positive $\mu$-measure, but for every $\lambda \in \Lambda_o$, the cyclic subgroup $\bZ\lambda.B$ has index at least two in $X$ (since $\bZ^r/\Lambda_o$ is not cyclic). Hence, $\mu(\bZ\lambda.B) \leq \frac{1}{2}$ for all $\lambda \in \Lambda_o$, and $B$ is not directionally expandable.
\end{example}

Example \ref{Ex2} highlights that the existence of a finite (non-cyclic) $\bZ^r$-factor obstructs directional expansion. Therefore, any meaningful results about the expansive properties of $\bZ^r$-actions should account for these finite factors. We address this by imposing a condition on the normalized spectral measure of the rational spectrum, requiring it to be sufficiently small. To elaborate on this, we introduce some notation. \\

If $(X,\mu)$ is an ergodic $\bZ^r$-space and $B \subset X$ is a $\mu$-measurable set with positive $\mu$-measure, there exists a unique finite and non-negative Borel measure $\sigma_B$ on the dual group $\widehat{\bZ^r} \simeq \bT^r$ such that:
\[
\mu(B \cap \lambda.B) = \int_{\widehat{\bZ^r}} \xi(\lambda) \, d\sigma_B(\xi), \quad \textrm{for all $\lambda \in \bZ^r$}.
\]
It is a well-known fact (see e.g. Lemma \ref{LemmaSpectralMeas} below) that if $\mu$ is a $\bZ^r$-ergodic measure, then $\sigma_B(\{1\}) = \mu(B)^2 > 0$, allowing us to define the \emph{normalized spectral measure} $\widetilde{\sigma}_B$ by
\[
\widetilde{\sigma}_B = \frac{\sigma_B}{\sigma_B(\{1\})}.
\]
The finite $\bZ^r$-factors of $\bZ^r \acts (X,\mu)$ are directly related to an important subset $\rat(\bZ^r)$ of $\widehat{\bZ^r}$ known as the rational spectrum, defined by
\[
\rat(\bZ^r) = \{ \xi \in \widehat{\bZ^r} \, : \, \xi|_{\Lambda_o} = 1, \enskip \textrm{for some finite-index subgroup $\Lambda_o < \bZ^r$} \}.
\]
Our first main dynamical result is stated as follows. 

\begin{theorem}
\label{ThmIntroErg}
Let $(X,\mu)$ be an ergodic $\bZ^r$-space and let $\eps_o \geq 0$. Suppose that $B \subset X$ is a $\mu$-measurable set with positive $\mu$-measure such that $\widetilde{\sigma}_B(\rat(\bZ^r) \setminus \{1\}) \leq \eps_o$. Then, for all $\eps > \eps_o$, there exists a primitive vector $\lambda_\eps \in \bZ^r$ such that
\[
\mu(\bZ\lambda_\eps.B) > 1-\eps.
\]
\end{theorem}

A $\bZ^r$-space is said to be \emph{totally ergodic} if every finite-index subgroup of $\bZ^r$ acts ergodically on $(X,\mu)$. In this context, it is evident that $\sigma_B(\rat(\bZ^r) \setminus {1}) = 0$ for any $\mu$-measurable set $B \subset X$ with positive $\mu$-measure. As a result, we obtain the following corollary.

\begin{corollary}
\label{CorTotErg}
If $(X,\mu)$ is a totally ergodic $\bZ^r$-space, then every $\mu$-measurable set $B \subset X$ with positive $\mu$-measure is directionally expandable.
\end{corollary}

\begin{remark}
In contrast to Theorem \ref{ThmIntroErg}, it is noteworthy that Theorem \ref{ThmIntroIntersections} imposes no spectral constraints on the set $B$. To derive Theorem \ref{ThmIntroIntersections} from Theorem \ref{ThmIntoComb}, we initially establish that for any $B$, there exists a finite-index subgroup $\Lambda_o < \bZ^r$ along with an ergodic component $\nu$ for the action $\Lambda_o \acts (X,\mu)$ such that $\widetilde{\sigma}_{\nu,B}(\rat(\Lambda_o) \setminus \{1\})$ is small, while $\nu(B)$ is comparable to $\mu(B)$. This constitutes a somewhat intricate step and is outlined in Section \ref{sec:RatSpec}.
\end{remark}

\subsection{Organization of the paper}
In Section \ref{sec:prel}, basic concepts essential for the proofs are introduced. Section \ref{sec:ThmMainErg} establishes a more comprehensive version of Theorem \ref{ThmIntroErg}. The subsequent Section \ref{sec:RatSpec} delves into the reduction of the size of the rational spectrum upon passing to finite-index subgroups, a crucial step elucidated further in Section \ref{sec:ThmMainInt} where a more generalized form of Theorem \ref{ThmIntroIntersections} is proven. In Section \ref{sec:PrfComb} we provide a proof for Theorem \ref{ThmIntoComb}.

\subsection{Acknowledgements}
M.B. was supported by the grant 11253320 from the Swedish Research Council. A.F. was supported by the ARC via grants DP210100162 and DP240100472. A.F. is grateful to Kamil Bulinski for useful discussions on the topic of the paper, Mumtaz Hussain, Uri Onn and their institutions LaTrobe Bendigo and the ANU for the hospitality. A.F. also would like to thank Martin Wechselberger and all SDG community for stimulating discussions related to the content of this paper.

\section{Preliminaries}
\label{sec:prel}

\subsection{Free abelian groups and their duals}

Let $\Lambda$ be a free abelian group of rank $r$. We note that, upon fixing a $\bZ$-basis $\mathscr{B} = (\beta_1,\ldots,\beta_r)$ of $\Lambda$, the map
\[
\bZ^r \ra \Lambda, \enskip (m_1,\ldots,m_r) \mapsto \sum_{k=1}^r m_k \beta_k
\]
is a group isomorphism. For a positive integer $n \geq 1$, let $\Lambda(n) = n \cdot \Lambda$. Note that $\Lambda(n)$ has finite index in $\Lambda$. We will need the following basic result which is an immediate consequence of Schmidt's normal form.

\begin{lemma}
\label{LemmaFactsFAGs}
Let $n \geq 1$ and suppose that $\Lambda' < \Lambda(n)$ is a finite-index subgroup.
Then there is an integer $N \geq n$ such that $\Lambda(N) < \Lambda'$.
\end{lemma}

\subsubsection{Primitive vectors}

A vector $\lambda \in \Lambda$ is \emph{primitive} if there are elements $\lambda_2,\ldots,\lambda_r \in \Lambda$ such that $\{\lambda,\ldots,\lambda_r\}$ is a $\bZ$-basis of $\Lambda$. If we fix a basis $\mathscr{B} = \{\beta_1,\ldots,\beta_r\}$ and write
\[
\lambda = \sum_{k=1}^r m_k \beta_k,
\]
then it is well-known (see e.g. \cite[Theorem 32]{Siegel}) that $\lambda$ is primitive if an only if $\gcd(m_1,\ldots,m_r) = 1$. We denote the set of primitive vectors in $\Lambda$ by $\cP_\Lambda$.

\subsubsection{Dual groups and rational spectrum}

Let $S^1 \subset \bC^*$ denote the (multiplicative) circle group and let $\widehat{\Lambda} = \Hom(\Lambda,S^1)$ denote the dual group of $\Lambda$. Note that $\widehat{\Lambda}$ is a closed (hence compact and metric) subgroup of the countable product $(S^1)^\Lambda$. We define the \emph{rational spectrum $\rat(\Lambda)$} of $\widehat{\Lambda}$ by
\[
\rat(\Lambda) = \{ \xi \in \widehat{\Lambda} \, : \, \xi|_{\Lambda_o} = 1 \enskip \textrm{for some finite-index subgroup $\Lambda_o < \Lambda$} \}.
\]
Note that the trivial character $1$ always belongs to $\rat(\Lambda)$. If $H$ is a subgroup of $\Lambda$, we write $H^\perp$ for its \emph{annihilator}, defined by
\[
H^\perp = \{ \xi \in \widehat{\Lambda} \, : \, \xi(\lambda) = 1 \enskip \textrm{for all $\lambda \in H$} \}.
\] 
Note that $H^\perp$ is always a closed (hence compact) subgroup of $\widehat{\Lambda}$.
If $\lambda \in \cP_\Lambda$, we denote by $L_\lambda = \bZ \lambda$ the cyclic subgroup generated by $\lambda$. Note that 
\[
L_\lambda^{\perp} = \{ \xi \in \widehat{\Lambda} \, : \, \xi(\lambda) = 1 \}.
\]

\subsubsection{Haystacks}

\begin{definition}[Haystack]
An infinite subset $\cH \subset \cP_\Lambda$ is a \emph{haystack} if for every $r$-tuple $(\lambda_1,\ldots,\lambda_r)$ of distinct elements in $\cH$, the subgroup $\linspan_{\bZ}(\lambda_1,\ldots,\lambda_r)$ has finite index in $\Lambda$.
\end{definition}

\begin{remark}
\label{RmkHaystack}
Note that if $\cH$ is a haystack, then
\[
\bigcap_{\lambda \in F} L_{\lambda}^\perp \subseteq \rat(\Lambda), \quad \textrm{for every $F \subset \cH$ with $|F| \geq r$}.
\]
Indeed, if $\lambda_1,\ldots,\lambda_r$ generates a finite-index subgroup $\Lambda_o < \Lambda$, then 
\[
\bigcap_{k=1}^r L_{\lambda_k}^\perp = \widehat{\Lambda}_o^{\perp},
\]
and $\widetilde{\Lambda}_o^{\perp}$ is clearly contained in $\rat(\Lambda)$.
\end{remark}

Haystacks can be produced in many different ways. The next lemma provides an explicit construction. 

\begin{lemma}
Let $\mathscr{B} = \{\beta_1,\ldots,\beta_r\}$ be a basis of $\Lambda$ and let 
$1 < m_1 < \ldots < m_r$ be relatively prime integers. Then,
\[
\cH = \Big\{ \sum_{k=1}^r m_k^n \cdot \beta_k \, : \, n \geq 1 \Big\} 
\]
is a haystack in $\cP_\Lambda$. 
\end{lemma}

\begin{proof}
First note that since $m_1,\ldots,m_r$ are relatively prime, $\cH$ is contained in $\cP_\Lambda$. Thus, to prove that $\cH$ is a haystack, it is enough to show 
that every set of $r$ distinct elements $\lambda_1,\ldots,\lambda_r \in \cH$ is linearly independent over $\bR$. \\

Fix an $r$-tuple $\lambda_1,\ldots,\lambda_r \in \cH$, and let $q_1,\ldots,q_r$ be real numbers such that
\begin{equation}
\label{qzero}
\sum_{j=1}^r q_j \lambda_j = 0.
\end{equation}
We want to show that $q_1 = \ldots = q_r = 0$. After possibly permuting indices, we can find 
integers $1 \leq n_1 < n_2 < \ldots < n_r$ such that
\[
\lambda_j = \sum_{k=1}^r m_k^{n_j} \beta_k, \quad \textrm{for all $j=1,\ldots,r$}.
\]
Then, since $\{\beta_1,\ldots,\beta_r\}$ is a basis, it follows from \eqref{qzero} that
\[
\sum_{j=1}^r q_j m_k^{n_j} = 0, \quad \textrm{for all $k=1,\ldots,r$},
\]
or equivalently, the vector $q = (q_1,\ldots,q_r)$ belongs to the (left) kernel of the $r \times r$-matrix 
\[
A = \{ e^{n_j \ln(m_k)} \}_{j,k=1}^r.
\]
However, by \cite[Chapter 1, Paragraph 2]{K}, the kernel $(x,y) \mapsto e^{xy}$ is strictly totally positive on $\bR \times \bR$, and thus $\det(A) > 0$ for all $m_1 < \ldots < m_r$ and $n_1 < \ldots < n_r$. We conclude that $q = 0$.
\end{proof}

\subsection{Ergodic theory of free abelian groups}

We say that a standard Borel probability space $(X,\mu)$ is a \emph{$\Lambda$-space}
if $X$ is equipped with a measurable $\Lambda$-action which preserves $\mu$, and we
say that $(X,\mu)$ is an \emph{ergodic $\Lambda$-space} if $\mu$ is also ergodic with
respect to this action.

\subsubsection{Spectral measures and normalized spectral measures}

Let $(X,\mu)$ be an ergodic $\Lambda$-space and let $B \subset X$ be a $\mu$-measurable set with positive $\mu$-measure. Bochner's Theorem tells us that there is a unique finite and non-negative Borel measure $\sigma_B$ on $\widehat{\Lambda}$ such that
\[
\mu(B \cap \lambda.B) = \int_{\widehat{\Lambda}} \xi(\lambda) \, d\sigma_B(\xi), \quad \textrm{for all $\lambda \in \Lambda$}.
\]
If we wish to emphasize the dependence on the measure $\mu$, we write $\sigma_{\mu,B}$. 
Recall that if $\lambda \in \Lambda$, then $L_\lambda$ denotes the cyclic subgroup 
$\bZ \lambda < \Lambda$.
\begin{lemma}[Spectral measures and annihilators]
\label{LemmaSpectralMeas}
Let $(X,\mu)$ be an ergodic $\Lambda$-space and let $B \subset X$ be a $\mu$-measurable set with positive $\mu$-measure. Then, $\sigma_B(\{1\}) = \mu(B)^2$ and $\sigma_B(\widehat{\Lambda}) = \mu(B)$, and for every $\lambda \in \Lambda$, 
\[
\sigma_{B}(L_\lambda^{\perp}) = \int_X \bE_\mu[\chi_B \, | \, \cE_{\bZ \lambda}]^2 \, d\mu,
\]
where $\cE_{\bZ \lambda}$ denotes the sub-$\sigma$-algebra of $\mathscr{B}_X$ consisting of $\mu$-almost $\bZ\lambda$-invariant subsets.
\end{lemma}

\begin{proof}
First note that since $\xi(0) = 1$ for all $\xi \in \widehat{\Lambda}$, we have
\[
\mu(B) = \mu(B \cap 0.B) = \int_{\widehat{\Lambda}} 1 \, d\sigma_B(\xi) = \sigma(\widehat{\Lambda}).
\]
Let $(F_n)$ be a F\o lner sequence in $\Lambda$. Then, by the weak ergodic theorem,
\[
\lim_{n \ra \infty} \frac{1}{|F_n|} \sum_{\lambda \in F_n} \mu(B \cap \lambda.B) = \mu(B)^2,
\]
and, for $\xi \in \widehat{\Lambda}$,
\[
\lim_{n \ra \infty} \frac{1}{|F_n|} \sum_{\lambda \in F_n} \xi(\lambda) = 
\left\{ 
\begin{array}{ll}
1 & \textrm{if $\xi = 1$} \\[0.1cm]
0 & \textrm{if $\xi \neq 1$}
\end{array}
\right..
\]
Hence, by dominated convergence,
\[
\lim_{n \ra \infty} \frac{1}{|F_n|} \sum_{\lambda \in F_n} \mu(B \cap \lambda.B) 
= \lim_{N \ra \infty} \int_{\widehat{\Lambda}} \Big( \frac{1}{|F_n|} \sum_{\lambda \in F_n} \xi(\lambda) \Big) \, d\sigma_B(\xi) = \sigma_B(\{1\}),
\]
which shows that $\sigma_B(\{1\}) = \mu(B)^2$. Fix $\lambda \in \Lambda$ and consider the restriction of the action to the subgroup $L_\lambda = \bZ \lambda$. We denote by 
$\cE_{\bZ \lambda}$ the sub-$\sigma$-algebra of $\mathscr{B}_X$ consisting of $\mu$-almost $\bZ\lambda$-invariant subsets. Geometric summation tells us that
\[
\lim_{n \ra \infty} \frac{1}{n} \sum_{k=0}^{n-1} \xi(k\lambda) = 
\lim_{n \ra \infty} \frac{1}{n} \sum_{k=0}^{n-1} \xi(\lambda)^k = 
\left\{ 
\begin{array}{ll}
1 & \textrm{if $\xi(\lambda) = 1$} \\[0.1cm]
0 & \textrm{if $\xi(\lambda) \neq 1$},
\end{array}
\right.,
\]
or equivalently,
\[
\lim_{n \ra \infty} \frac{1}{n} \sum_{k=0}^{n-1} \xi(k\lambda) = \chi_{L_\lambda^{\perp}}(\xi), \quad \textrm{for all $\xi \in \Lambda^\perp$}.
\]
We conclude, by dominated convergence, 
\[
\sigma_{B}(L_\lambda^\perp) = \lim_{n \ra \infty} \frac{1}{n} \sum_{k=0}^{n-1} \mu(B \cap k\lambda.B), \quad \textrm{for all $\lambda \in \Lambda$}.
\]
On the other hand, by the weak ergodic theorem and using that conditional expectations are projections on $L^2(X,\mu)$,
\[
\lim_{n \ra \infty} \frac{1}{n} \sum_{k=0}^{n-1} \mu(B \cap k\lambda.B)
= \langle \chi_B, \bE_\mu[\chi_B | \cE_{\bZ \lambda}] \rangle_{L^2(X,\mu)} 
= \int_{X} \bE_\mu[\chi_B | \cE_{\bZ \lambda}]^2 \, d\mu,
\]
which finishes the proof.
\end{proof}

\subsubsection{Ergodic sets}

\begin{definition}[Ergodic set]
A set $S \subset \bZ$ is \emph{ergodic} if there exists an increasing 
sequence $(S_N)$ of finite subsets of $S$ such that for every $\bZ$-space $(Y,\nu)$
and $f \in \mathscr{L}^\infty(Y)$, there is a $\nu$-conull set $Y_f \subset Y$
such that
\[
\lim_{N \ra \infty} \frac{1}{|S_N|} \sum_{m \in S_N} f(m.y) = \bE_\nu[f \, | \, \cE](y), \quad \textrm{for all $y \in Y_f$},
\]
where $\cE$ denotes the sub-$\sigma$-algebra of $\nu$-almost $\bZ$-invariant sets.
\end{definition}

Birkhoff's ergodic theorem tells us that $\bZ$ is an ergodic set. However, there are also quite sparse ergodic subsets of $\bZ$ (see for instance \cite{BW} for a plethora of examples).

\section{Proof of Theorem \ref{ThmIntroErg}}

Let $\Lambda$ be a free abelian group of rank $r$ and let $(X,\mu)$ be an ergodic $\Lambda$-space. 
Let $B \subset X$ be a $\mu$-measurable set with positive $\mu$-measure. We recall that the normalized spectral measure $\widetilde{\sigma}_B$ is defined as
\[
\widetilde{\sigma}_B = \frac{\sigma_B}{\sigma_B(\{1\})},
\]
where $\sigma_B$ is the unique finite and non-negative measure on $\widehat{\Lambda}$
which satisfies
\[
\mu(B \cap \lambda.B) = \int_{\widehat{\Lambda}} \xi(\lambda) \, d\sigma_B(\xi), \quad \textrm{for all $\lambda \in \Lambda$}.
\]
If we wish to emphasize the dependence on the measure $\mu$, we write $\widetilde{\sigma}_{\mu,B}$. \\

In this section we prove the following generalization of Theorem \ref{ThmIntroErg}. Recall that $\cP_\Lambda$ denotes the set of primitive vectors in $\Lambda$. For the definitions of haystack, ergodic set and $\rat(\Lambda)$, we refer the reader to Section \ref{sec:prel}.

\begin{theorem}
\label{ThmMainErg}
Suppose that $\widetilde{\sigma}_{B}(\rat(\Lambda) \setminus \{1\}) \leq \eps_o$. Then, for all $\eps > \eps_o$ and for every haystack $\cH \subset \cP_\Lambda$, 
there exists $\lambda_\eps \in \cH$ such that
\[
\mu(S\lambda_\eps.B) > 1 - \eps,
\]
for every ergodic set $S \subset \bZ$. 
\end{theorem}

In the next subsection we will show how Theorem \ref{ThmMainErg} can be deduced from the following two lemmas, whose proofs are given at the end of this section. Recall
that if $\lambda \in \cP_\Lambda$, then $L_\lambda$ denotes the cyclic subgroup $\bZ \lambda$, and $L_\lambda^{\perp}$ its annihilator in $\widehat{\Lambda}$, i.e. 
\[
L_\lambda^{\perp} = \{ \xi \in \widehat{\Lambda} \, : \, \xi(\lambda) = 1 \}.
\]

\begin{lemma}[Expansion and spectral measures]
\label{LemmaExpansionSpectral}
For every $\mu$-measurable set $B \subset X$ with positive $\mu$-measure and for 
every $\lambda \in \Lambda$, 
\[
\mu(S\lambda.B) \geq \frac{1}{\widetilde{\sigma}_B(L_\lambda^{\perp})},
\]
for every ergodic set $S \subset \bZ$.
\end{lemma}

\begin{lemma}[Haystacks and small annihilators]
\label{LemmaHaystacksAnnihilators}
Let $\tau$ be a finite Borel measure on $\widehat{\Lambda}$ such that $\tau(\rat(\Lambda)) = 0$, and let $\cH \subset \cP_\Lambda$ be a haystack. 
Then, for every $\delta > 0$, there exists $\lambda \in \cH$ such that
$\tau(L_\lambda^{\perp}) < \delta$.
\end{lemma}

\subsection{Proof of Theorem \ref{ThmMainErg} assuming Lemma \ref{LemmaExpansionSpectral} and Lemma \ref{LemmaHaystacksAnnihilators}}
\label{sec:ThmMainErg}
Fix $\eps > \eps_o \geq 0$ and suppose that $B \subset X$ is a $\mu$-measurable set with positive $\mu$-measure such that $\widetilde{\sigma}_B(\rat(\Lambda) \setminus \{1\}) \leq \eps_o$. We can thus write
\[
\widetilde{\sigma}_B = \delta_{1} + \alpha + \tau,
\]
where $\alpha$ is a finite measure supported on $\rat(\Lambda) \setminus \{1\}$ such that $\alpha(\rat(\Lambda) \setminus \{1\}) \leq \eps_o$ and $\tau$ is a finite measure such that $\tau(\rat(\Lambda)) = 0$. \\

Let $S \subset \bZ$ be an ergodic set. By Lemma \ref{LemmaExpansionSpectral}, 
\[
\mu(S\lambda.B) \geq \frac{1}{\widetilde{\sigma}_B(L_\lambda^{\perp})}, \quad \textrm{for all $\lambda \in \Lambda$}.
\]
In particular, using the decomposition of $\widetilde{\sigma}_B$ above, we have
\[
\mu(S\lambda.B) \geq \frac{1}{1 + \eps_o + \tau(L_\lambda^{\perp})}, \quad \textrm{for all $\lambda \in \Lambda$}.
\]
Pick 
\[
0 < \delta < \frac{1-(1+\eps_o)(1-\eps)}{1-\eps}.
\]
Let $\cH$ be a haystack. By Lemma \ref{LemmaHaystacksAnnihilators}, we can find $\lambda \in \cH$ such that
$\tau(L_\lambda^{\perp}) < \delta$, and thus
\[
\mu(S\lambda.B) \geq \frac{1}{1+\eps_o + \delta} > 1 - \eps,
\]
which finishes the proof (with $\lambda_\eps = \lambda$).

\subsection{Proof of Lemma \ref{LemmaExpansionSpectral}}

Let $B \subset X$ be a $\mu$-measurable set with positive $\mu$-measure. 
We will need the following lemma. 

\begin{lemma}
Let $S \subset \bZ$ be an ergodic set. Then, for every $\lambda \in \Lambda$, 
\[
\mu(S\lambda.B) \geq \mu(\{ \bE_\mu[\chi_B \, | \, \cE_{\lambda \bZ}] > 0\}),
\]
where $\cE_{\bZ \lambda}$ denotes the sub-$\sigma$-algebra of $\mathscr{B}_X$ consisting of $\mu$-almost $\bZ\lambda$-invariant subsets.
\end{lemma}

\begin{proof}
Fix $\lambda \in \Lambda$ and consider the action $\bZ \lambda \acts (X,\mu)$. Since $S$ is an ergodic set, there is a sequence $(S_n)$ of finite subsets of $S$ such that
\[
\bE_\mu[\chi_B | \cE_{\bZ \lambda}](x) 
=
\lim_{n \ra \infty} \frac{1}{|S_n|} \sum_{m \in S_n} \chi_B((-m\lambda).x), \quad \textrm{$\mu$-almost everywhere}.
\]
By Egorov's Theorem, for every $\eps > 0$, there exist an integer $N_\eps$ and a 
$\mu$-measurable set $X_\eps \subset X$ with $\mu(X_\eps) > 1-\eps$ such that
\[
\Big| \bE_\mu[\chi_B | \cE_{\bZ \lambda}](x) - \frac{1}{|S_n|} \sum_{m \in S_n} \chi_B((-m\lambda).x) \Big| < \eps, \quad \textrm{for all $n \geq N_\eps$ and $x \in X_\eps$}.
\]
In what follows, we write
\[
\bE_n(x) = \frac{1}{|S_n|} \sum_{m \in S_n} \chi_B((-m\lambda).x).
\]
Fix $\eps > 0$ and $\delta > \eps$. Then,
\begin{align*}
\mu(\{ \bE_\mu[\chi_B | \cE_{\bZ \lambda}] \geq \delta \})
&=
\mu(\{ \bE_\mu[\chi_B | \cE_{\bZ \lambda}] - \bE_n + \bE_n  \geq \delta \} \cap X_\eps) \\[0.2cm]
&+
\mu(\{ \bE_\mu[\chi_B | \cE_{\bZ \lambda}] \geq \delta \} \cap X^c_\eps) \leq 
\mu(\{ \bE_n > \delta-\eps \}) + \eps \\[0.2cm]
&\leq \mu(\{ \bE_n > 0 \}) + \eps,
\end{align*}
for all $n \geq N_\eps$. Note that
\[
\mu(\{ \bE_n > 0 \}) = \mu(S_n\lambda.B) \leq \mu(S\lambda.B).
\]
Hence, since $\delta > \eps > 0$ are arbitrary, we conclude that
\[
\mu(\{ \bE_\mu[\chi_B | \cE_{\bZ \lambda}] > 0 \}) \leq \mu(S\lambda.B),
\]
which finishes the proof.
\end{proof}

To prove Lemma \ref{LemmaExpansionSpectral}, we need to estimate $\mu(\{ \bE_\mu[\chi_B | \cE_{\bZ \lambda}] > 0 \})$. To do this, we use the following trick. Suppose $f$ is a non-negative $\mu$-measurable function on $X$. Then, by the Cauchy-Schwarz inequality,
\[
\int_X f \, d\mu = \int_X f \chi_{\{f > 0\}} \, d\mu \leq \Big(\int_X f^2 \, d\mu\Big)^{1/2} \mu(\{f > 0\})^{1/2},
\]
and thus
\[
\mu(\{f > 0\}) \geq \frac{\Big(\int_X f \, d\mu \Big)^2}{\int_X f^2 \, d\mu}.
\]
If we apply this inequality to $f = \bE_\mu[\chi_B | \cE_{\bZ \lambda}]$, we see that
\[
\mu(\{ \bE_\mu[\chi_B | \cE_{\bZ \lambda}] > 0 \}) \geq 
\frac{\Big( \int_X \bE_\mu[\chi_B | \cE_{\bZ \lambda}] \, d\mu \Big)^2}{\int_X \bE_\mu[\chi_B | \cE_{\bZ \lambda}]^2 \, d\mu} = \frac{\mu(B)^2}{\int_X \bE_\mu[\chi_B | \cE_{\bZ \lambda}]^2 \, d\mu}.
\]
By Lemma \ref{LemmaSpectralMeas},
\[
\sigma_B(\{1\}) = \mu(B)^2 \qand \sigma_B(L_\lambda^{\perp}) = \int_X \bE_\mu[\chi_B | \cE_{\bZ \lambda}]^2 \, d\mu,
\]
so we conclude that
\[
\mu(\{ \bE_\mu[\chi_B | \cE_{\bZ \lambda}] > 0 \}) \geq \frac{\sigma_B(\{1\})}{\sigma_{B}(L_\lambda^\perp)} = \frac{1}{\widetilde{\sigma}_B(L_\lambda^{\perp})}.
\]

\subsection{Proof of Lemma \ref{LemmaHaystacksAnnihilators}}

Let $\tau$ be a finite probability measure on $\widehat{\Lambda}$ such that $\tau(\rat(\Lambda)) = 0$ and let $\cH \subset \cP_\Lambda$ be a haystack. Fix an enumeration $(\lambda_n)$ of the elements in $\cH$. Since $\cH$ is a haystack and $\tau(\rat(\Lambda)) = 0$, we have, in view of Remark \ref{RmkHaystack},
\[
\tau\Big(\bigcap_{n \in F} L_{\lambda}^\perp \Big) = 0, \quad \textrm{for every $F \subset \cH$ with $|F| \geq r$}.
\]
Let $\delta > 0$ and pick an integer $N \geq 1$ such that $r \cdot \tau(\widehat{\Lambda})/N < \delta$. 
Lemma \ref{LemmaHaystacksAnnihilators} is now an immediate consequence of the following general measure-theoretic result, applied to
\[
(Y,\nu) = (\widehat{\Lambda},\tau) \qand A_n = L_{\lambda_n}^\perp \qand p = r.
\]
\begin{lemma}
\label{LemmaSmallsetsFromSmallIntersections}
Let $(Y,\nu)$ be a finite measure space and let $p$ be a positive integer. 
Then, for every positive integer $N$ and for every sequence $(A_n)$ of $\nu$-measurable sets in $Y$ such that 
\begin{equation}
\label{AssIntersections}
\nu\Big( \bigcap_{n \in F} A_n \Big) = 0, \quad \textrm{for every $F \subset \bN$ with $|F| \geq p$},
\end{equation}
there exists an index $n \leq N$ such that $\nu(A_n) < p \cdot \nu(Y)/N$.
\end{lemma}

\begin{proof}
We can without loss of generality assume that $\nu(Y) = 1$. Let us fix a positive integer $N$, and let $(A_n)$ be a sequence of $\nu$-measurable sets in $Y$ which satisfies \eqref{AssIntersections}. We assume, for the sake of contradiction, that
\[
\nu(A_n) \geq \frac{p}{N}, \quad \textrm{for all $1 \leq n \leq N$}.
\]
Then, 
\[
\int_Y \left( \frac{1}{N} \sum_{n=1}^N \chi_{A_n} \right) \, d\nu \geq \frac{p}{N},
\]
and thus the set
\[
C_N = \left\{ x \in X \, : \, \sum_{n=1}^N \chi_{A_n}(x) \geq p \right\}
\]
has positive $\nu$-measure. Define $\cF_{p,N} = \{ F \subseteq \{1,\ldots,N\} \, : \, |F| \geq p \}$ and note that
\begin{equation}
\label{union}
C_N = \bigcup_{F \in \cF_{p,N}} C_{N,F},
\end{equation}
where
\[
C_{N,F} = \{ x \in X \, : \, x \in A_n, \enskip \textrm{for all $n \in F$} \} = \bigcap_{n \in F} A_n.
\]
Since we assume that $\nu(C_N) > 0$, it follows from \eqref{union} that there must be at least one subset $F \subset \{1,\ldots,N\}$ with $|F| \geq p$ such that 
\[
\nu(C_{N,F}) = \nu\Big( \bigcap_{n \in F} A_n \Big) > 0,
\]
which contradicts \eqref{AssIntersections}.
\end{proof}

\section{Rational spectrum and ergodic decomposition}
\label{sec:RatSpec}

Let $\Lambda$ be a free abelian group of rank $r$ and let $(X,\mu)$ be an ergodic $\Lambda$-space. Let $B \subset X$ be a $\mu$-measurable set with positive $\mu$-measure, and let $\Lambda(n) = n \cdot \Lambda$ for $n \geq 1$. Note that $\Lambda(n)$
has finite index in $\Lambda$. \\

In the previous section we have seen that we can achieve significant directional expansion of the set $B$ if $\widetilde{\sigma}_B(\rat(\Lambda) \setminus \{1\})$ is small enough. However, for the applications to volume spectra that we have in mind, this is not a natural assumption. The aim of this section is to show that upon passing to a small enough finite index subgroup $\Lambda_o$ of $\Lambda$, we can always select a good ergodic component $\nu_o$ of the sub-action of $\Lambda_o$, for which the normalized spectral measure $\widetilde{\sigma}_{\nu_o,B}(\rat(\Lambda_o) \setminus \{1\})$ of 
the rational spectrum is small. \\

The exact statement reads as follows.

\begin{proposition}[Shrinking the rational spectrum]
\label{PropositionShrinkingRationalSpec}
For every $\eps_o > 0$, there exist an integer $n$, a positive constant $c$, and a $\Lambda(n)$-invariant and $\Lambda(n)$-ergodic probability measure $\nu$ on $X$ such that either
\[
\nu(B) \geq 1/3 \quad \textrm{or} \quad \mu(B) < 3 \cdot \nu(B)
\] 
and
\[
\widetilde{\sigma}_{\nu,B}(\rat(\Lambda(n)) \setminus \{1\}) < \eps_o,
\]
and
\[
\mu\Big(\bigcap_{\lambda \in F} \lambda.B \Big)
\geq
c \cdot \nu\Big(\bigcap_{\lambda \in F} \lambda.B  \Big),
\]
for every $F \subset \Lambda(n)$.
\end{proposition}

\begin{remark}
We stress that the integer $n$, the measure $\nu$ and the constant $c$ will in general depend on the set $B$.
\end{remark}

Proposition \ref{PropositionShrinkingRationalSpec} will be a consequence of the following technical lemma, whose proof will be given at the end of this section. 
To explain its statement, note that if $\Lambda_o < \Lambda$ is any subgroup, then
there is always a continuous restriction homomorphism
\[
\pi_o : \widehat{\Lambda} \ra \widehat{\Lambda}_o, \enskip \xi \mapsto \xi|_{\Lambda_o}.
\]
In particular, $\pi_o^{-1}(\rat(\Lambda_o)) \subset \rat(\Lambda)$.

\begin{lemma}[Rational spectrum and finite-index subgroups]
\label{LemmaBehaviourErgCpts}
For every finite-index subgroup $\Lambda_o < \Lambda$ there exist a $\Lambda_o$-invariant and $\Lambda_o$-ergodic probability measure $\nu$ on $X$ and a positive constant $c$ such that either 
\[
\nu(B) \geq 1/3 \quad \textrm{or} \quad \mu(B) < 3 \cdot \nu(B)
\]
and
\[
\widetilde{\sigma}_{\nu,B}\big(\rat(\Lambda_o) \setminus \{1\}\big) \leq 3 \cdot 
\widetilde{\sigma}_{\mu,B}\big(\pi_o^{-1}\big(\rat(\Lambda_o) \setminus \{1\}\big)\big)  
\]
and
\[
\mu\Big(\bigcap_{\lambda \in F} \lambda.B \Big)
\geq
c \cdot \nu\Big(\bigcap_{\lambda \in F} \lambda.B  \Big),
\]
for every $F \subset \Lambda_o$. 
\end{lemma}

\subsection{Proof of Proposition \ref{PropositionShrinkingRationalSpec} assuming Lemma \ref{LemmaBehaviourErgCpts}}

Fix $\eps_o \geq 0$ and a $\mu$-measurable set $B \subset X$ with positive $\mu$-measure. By Lemma \ref{LemmaBehaviourErgCpts}, there are, for every $n$, a $\Lambda(n)$-invariant and $\Lambda(n)$-ergodic probability measure $\nu_n$ on $X$, and a positive constant $c_{n}$ such that either 
\[
\nu_n(B) \geq 1/3 \quad \textrm{or} \quad \mu(B) < 3 \cdot \nu_n(B)
\]
and
\[
\widetilde{\sigma}_{\nu_n,B}(\rat(\Lambda(n)) \setminus \{1\}) < 3 \cdot 
\widetilde{\sigma}_{\mu,B}(\pi_n^{-1}(\rat(\Lambda(n)) \setminus \{1\})) 
\]
and
\[
\mu\Big(\bigcap_{\lambda \in F} \lambda.B  \Big) \geq c_{n} \cdot \nu_n\Big(\bigcap_{\lambda \in F} \lambda.B  \Big),
\]
for every $F \subset \Lambda(n)$, where $\pi_n : \widehat{\Lambda} \ra \widehat{\Lambda(n)}$ denotes the restriction map $\xi \mapsto \xi|_{\Lambda(n)}$. It thus suffices to show that we can find $n$ such that
\[
\widetilde{\sigma}_{\mu,B}(\pi_n^{-1}(\rat(\Lambda(n)) \setminus \{1\})) \leq \eps_o/3,
\]
in which case the theorem holds with $c = c_{n}$ and $\nu = \nu_n$. \\
 
Since $\widetilde{\sigma}_{\mu,B}$ is a finite non-negative measure on $\widehat{\Lambda}$, and the sequence of sets 
\[
A_m = \pi_{m!}^{-1}(\rat(\Lambda(m!)) \setminus \{1\}) \subset \widehat{\Lambda}, \quad m \geq 1,
\]
is decreasing, it is enough to prove that
\begin{equation}
\label{emptyintersection}
A_\infty := \bigcap_{m=1}^\infty A_m = \emptyset.
\end{equation}
Assume, for the sake of contradiction, that this intersection is non-empty, and 
pick an element $\xi \in A_\infty$. Then, for every $m \geq 1$, we have $\xi|_{\Lambda(m!)} \neq 1$, and there is a finite-index subgroup $\Lambda_m \subset \Lambda(m!)$ such that $\xi|_{\Lambda_m} = 1$. By Lemma \ref{LemmaFactsFAGs}, there
exists an integer $N$ such that $\Lambda(N) < \Lambda_m$. However, this implies that 
$\xi|_{\Lambda(N!)} = 1$ (or equivalently, $\xi \notin A_N$), which is a contradiction to the assumption that $\xi \in A_\infty$.

\subsection{Proof of Lemma \ref{LemmaBehaviourErgCpts}}

Let $\Lambda_o < \Lambda$ be a finite-index subgroup. Since $\mu$ is $\Lambda$-invariant and $\Lambda$-ergodic,
there exists, by \cite[Theorem 8.20]{EW}, a $\Lambda$-invariant and $\Lambda$-ergodic probability measure $\alpha$ on the (weak*) Borel set 
$\Prob^{\textrm{erg}}_{\Lambda_o}(X) \subset \Prob(X)$ such that
\[
\mu = \int_{\Prob_{\Lambda_o}^{\textrm{erg}}(X)} \nu \, d\alpha(\nu).
\]
Since $\Lambda_o$ acts trivially on $\Prob^{\textrm{erg}}_{\Lambda_o}(X)$, the action of $\Lambda$ on this space descends to an action by the finite group $\Lambda/\Lambda_o$. Since $\alpha$ is ergodic with respect to this action, we see that the support of $\alpha$ consists of a single $\Lambda/\Lambda_o$-orbit. In particular, the support of $\alpha$ is finite, so we conclude that there is a finite set $Q \subset \Prob_{\Lambda_o}^{\textrm{erg}}(X)$ with $|Q| \leq |\Lambda/\Lambda_o|$ and a strictly positive function $\alpha : Q \ra [0,1]$ such that
\begin{equation}
\label{ergcpt}
\sum_{\nu \in Q} \alpha(\nu) = 1 \qand \mu = \sum_{\nu \in Q} \alpha(\nu) \nu.
\end{equation}
Since $\mu(B) > 0$, the set $Q_B = \{ \nu \in Q \, : \, \nu(B) > 0 \}$ is non-empty.
Furthermore, 
\begin{equation}
\label{defc}
c := \min\{ \alpha(\nu) \, : \, \nu \in Q_B\} > 0,
\end{equation}
and for every $\nu \in Q_B$, it follows from \eqref{ergcpt} that
\begin{equation}
\mu(B') \geq c \cdot \nu(B') \quad \textrm{for every $\mu$-measurable $B' \subset B$}. 
\end{equation}
In particular, if $F \subset \Lambda_o$ and $\lambda_o \in F$, we can write
\[
\bigcap_{\lambda \in F} \lambda.B = \lambda_o.B', \quad \textrm{where 
$B'= \hspace{-0.3cm} \bigcap_{\lambda \in F-\lambda_o} \hspace{-0.2cm}\lambda.B \subset B$},
\]
and thus, for every $\nu \in Q_B$, using that both $\mu$ and $\nu$ are $\Lambda_o$-invariant, we see that
\begin{align}
\mu\Big(\bigcap_{\lambda \in F} \lambda.B\Big) 
&= 
\mu(\lambda_o.B') = \mu(B') \geq c \cdot \nu(B') \nonumber \\
&= c \cdot \nu(\lambda_o.B') = c \cdot \nu\Big(\bigcap_{\lambda \in F} \lambda.B\Big). \label{intersectionmunu}
\end{align} 
Note that for every $\lambda \in \Lambda_o$, 
\begin{align*}
\mu(B \cap \lambda.B) 
&= 
\int_{\widehat{\Lambda}} \xi(\lambda) \, d\sigma_{\mu,B}(\xi)
= \int_{\widehat{\Lambda}} \xi|_{\Lambda_o}(\lambda) \, d\sigma_{\mu,B}(\xi) \\[0.2cm]
&=
\int_{\widehat{\Lambda}_o} \eta(\lambda) \, d(\pi_o)_*\sigma_{\mu,B}(\eta),
\end{align*}
and
\begin{align*}
\mu(B \cap \lambda.B) 
&=
\sum_{\nu \in Q_B} \alpha(\nu) \cdot \nu(B \cap \lambda.B) = \sum_{\nu \in Q_B} \alpha(\nu) \cdot \int_{\widehat{\Lambda}_o} \eta(\lambda) \, d\sigma_{\nu,B}(\eta) \\[0.2cm]
&=
\int_{\widehat{\Lambda}_o} \eta(\lambda) d\sigma'(\eta), \quad \textrm{where $\sigma' = \sum_{\nu \in Q_B} \alpha(\nu) \cdot \sigma_{\nu,B}$}.
\end{align*}
Since finite measures on $\widehat{\Lambda}_o$ are uniquely determined by their Fourier transforms, we conclude that $(\pi_o)_*\sigma_{\mu,B} = \sigma'$, that is to say, 
\begin{equation}
\label{pioid}
(\pi_o)_*\sigma_{\mu,B} = \sum_{\nu \in Q_B} \alpha(\nu) \cdot \sigma_{\nu,B}. 
\end{equation}
In particular,
\[
\sigma_{\mu,B}(\pi_o^{-1}(\rat(\Lambda_o) \setminus \{1\})) = 0 \implies
\widetilde{\sigma}_{\nu,B}(\rat(\Lambda_o) \setminus \{1\}) = 0, \quad \textrm{for 
all $\nu \in Q_B$}.
\]
Hence, if pick any $\nu \in Q_B$ such that $\nu(B)$ is maximal, then $\mu(B) \leq \nu(B)$ and $\nu$ will satisfy all of the properties asserted in the lemma. \\

Let us from now on assume that $\widetilde{\sigma}_{\mu,B}(\pi_o^{-1}(\rat(\Lambda_o) \setminus \{1\})) > 0$. First note that
\begin{align}
\mu(B)^2 &= \left( \sum_{\nu \in Q_B} \alpha(\nu) \nu(B) \right)^2 \leq \left(\, \sum_{\nu \in Q_B} \alpha(\nu) \right) \cdot \left( \sum_{\nu \in Q_B} \alpha(\nu) \nu(B)^2 \right) \nonumber \\[0.2cm]
&\leq \sum_{\nu \in Q_B} \alpha(\nu) \nu(B)^2, \label{holder}
\end{align}
by H\"older's inequality. Hence, by \eqref{pioid} and \eqref{holder},
\begin{align}
(\pi_o)_*\widetilde{\sigma}_{\mu,B}
&=
\frac{(\pi_o)_*\sigma_{\mu,B}}{\mu(B)^2} 
=
\frac{\sum_{\nu \in Q_B} \alpha(\nu) \, \sigma_{\nu,B}}{\mu(B)^2} 
\geq \frac{\sum_{\nu \in Q_B} \alpha(\nu) \, \sigma_{\nu,B}}{ \sum_{\nu \in Q_B} \alpha(\nu) \nu(B)^2} \nonumber \\[0.2cm]
&=
\frac{\sum_{\nu \in Q_B} \alpha(\nu) \nu(B)^2 \cdot  \widetilde{\sigma}_{\nu,B}}{ \sum_{\nu \in Q_B} \alpha(\nu) \nu(B)^2} = \sum_{\nu \in Q_B} \beta(\nu) \, \widetilde{\sigma}_{\nu,B},
\label{normspecmeasmunu}
\end{align}
where $\beta : Q_B \ra [0,1]$ denotes the probability measure
\[
\beta(\nu) = \frac{\alpha(\nu) \nu(B)^2}{ \sum_{\nu' \in Q_B} \alpha(\nu') \nu'(B)^2}, \quad \nu \in Q_B,
\]
and the inequality in \eqref{normspecmeasmunu} is interpreted in the sense of non-negative measures on $\widehat{\Lambda}_o$. In particular, if we apply \eqref{normspecmeasmunu} to the sets $\rat(\Lambda_o) \setminus \{1\}$ and $\widehat{\Lambda}_o$ respectively, we get the inequalities
\begin{equation}
\label{ineqrat}
\widetilde{\sigma}_{\mu,B}(\pi_o^{-1}(\rat(\Lambda_o) \setminus \{1\}))
\geq \sum_{\nu \in Q_B} \beta(\nu) \widetilde{\sigma}_{\nu,B}(\rat(\Lambda_o) \setminus \{1\}),
\end{equation}
and, 
\begin{equation}
\label{ineqinvmunu}
\widetilde{\sigma}_{\mu,B}(\pi_o^{-1}(\widehat{\Lambda}_o)) = \widetilde{\sigma}_{\mu,B}(\widehat{\Lambda}) = \frac{1}{\mu(B)} \geq \sum_{\nu \in Q_B} \beta(\nu) \widetilde{\sigma}_{\nu,B}(\widehat{\Lambda}_o) = \sum_{\nu \in Q_B} \frac{\beta(\nu)}{\nu(B)},
\end{equation}
since, by Lemma \ref{LemmaSpectralMeas}, we have
\[
\widetilde{\sigma}_{\mu,B}(\widehat{\Lambda}) = \frac{\mu(B)}{\mu(B)^2} = \frac{1}{\mu(B)}
\qand 
\widetilde{\sigma}_{\nu,B}(\widehat{\Lambda}_o) = \frac{\nu(B)}{\nu(B)^2} = \frac{1}{\nu(B)}, \quad \textrm{for all $\nu \in Q_B$}.
\]
Let us now define the sets
\[
S_1 = \{ \nu \in Q_B \, : \, 
\widetilde{\sigma}_{\nu,B}(\rat(\Lambda_o) \setminus \{1\}) \geq 3 \cdot \widetilde{\sigma}_{\mu,B}(\pi_o^{-1}(\rat(\Lambda_o) \setminus \{1\})) \}
\]
and
\[
S_2 = \left\{ \nu \in Q_B \, : \, \frac{1}{\nu(B)} \geq \frac{3}{\mu(B)} \right\}.
\]
Then, since $\widetilde{\sigma}_{\mu,B}(\pi_o^{-1}(\rat(\Lambda_o) \setminus \{1\})) > 0$, the bounds \eqref{ineqrat} and \eqref{ineqinvmunu}, combined with Markov's inequality, tell us that
\[
\beta(S_1) \leq 1/3 \qand \beta(S_2) \leq 1/3,
\]
and thus $T := S_1^c \cap S_2^c$ is non-empty. For every $\nu \in T$, we have
\[
\widetilde{\sigma}_{\nu,B}(\rat(\Lambda_o) \setminus \{1\}) < 3 \cdot \widetilde{\sigma}_{\mu,B}(\pi_o^{-1}(\rat(\Lambda_o) \setminus \{1\}))
\qand
\mu(B) < 3 \cdot \nu(B).
\]
The last inequality is only non-trivial if $\nu(B) < 1/3$ (the first inequality can still be non-trivial even if $\nu(B) \geq 1/3$ for all $\nu \in T$). We conclude that any $\nu$ in $T$ will satisfy the properties asserted in the lemma, and we are done.

\section{Proof of Theorem \ref{ThmIntroIntersections}}
\label{sec:ThmMainInt}

Let $\Lambda$ be a free abelian group of rank $r$ and let $(X,\mu)$ be an ergodic $\Lambda$-space. Let $B \subset X$ be a $\mu$-measurable set with positive $\mu$-measure. In this section, we will prove the following generalization of Theorem \ref{ThmIntroIntersections}.

\begin{theorem}
\label{ThmMainIntersections}
For every $p \geq 2$, there is a positive integer $n$, such that for every haystack $\cH \subset \cP_\Lambda$ and ergodic set $S \subset \bZ$, there exist $\lambda \in \cH$ and $m_1 \in S$ with the property that for all $\lambda_2,\ldots,\lambda_p \in \Lambda$, there are $m_2,\ldots,m_p \in S$ such that
\[
\mu\Big( B \cap m_1n\lambda.B \cap \Big(\bigcap_{k=2}^p (m_kn \lambda + n \lambda_k).B \Big)\Big) > 0.
\]
\end{theorem}

\begin{proof}
Let us fix an integer $p \geq 2$, a $\mu$-measurable set $B \subset X$ with positive $\mu$-measure and $\eps_o > 0$ to be chosen later. By Proposition \ref{PropositionShrinkingRationalSpec}, 
we can find an integer $n$, a positive constant $c$ and a $\Lambda(n)$-invariant and $\Lambda(n)$-ergodic probability measure $\nu$ on $X$ such that either $\nu(B) \geq 1/3$
or $\mu(B) < 3 \cdot \nu(B)$, 
\begin{equation}
\label{lowbnd1}
\widetilde{\sigma}_{\nu,B}(\rat(\Lambda(n)) \setminus \{1\}) < \eps_o,
\end{equation}
and
\begin{equation}
\label{lowbnd2}
\mu\Big(\bigcap_{\lambda \in F} \lambda.B \Big)
\geq
c \cdot \nu\Big(\bigcap_{\lambda \in F} \lambda.B  \Big),
\end{equation}
for every $F \subset \Lambda(n)$. Fix a haystack $\cH \subset \cP_{\Lambda}$ and an ergodic set $S \subseteq \bZ$. Note that the set $\cH(n) := n \cdot \cH$ is a haystack in $\cP_{\Lambda(n)}$. Fix $\eps > \eps_o$. By Theorem \ref{ThmMainErg}, applied to the action $\Lambda(n) \acts (X,\nu)$, we can find $\lambda_\eps \in \cH(n)$ such that
\begin{equation}
\label{nulowbnd}
\nu(S\lambda_\eps.B) > 1-\eps. 
\end{equation}
Note that $\lambda_\eps = n\lambda$ for some $\lambda \in \cH$. We first claim that there exists $m_1 \in S$ such that
\begin{equation}
\label{bigint}
\nu(B \cap m_1\lambda_\eps.B) > \frac{\nu(B)^2}{2}.
\end{equation}
Indeed, since $S$ is an ergodic set, there is an increasing sequence $(S_N)$ of finite subsets of $S$, such that 
\[
\lim_{N \ra \infty} \frac{1}{|S_N|} \sum_{m \in S_N} \nu(B \cap m \lambda.B) = \int_{X} \bE_\nu[\chi_B \, | \, \cE_{\bZ \lambda_\eps}]^2 \, d\nu \geq \nu(B)^2,
\]
where $\cE_{\bZ \lambda_\eps}$ denotes the sub-$\sigma$-algebra of $\mathscr{B}_X$
consisting of $\nu$-almost $\bZ \lambda_\eps$-invariant sets. This readily implies \eqref{bigint}. Let us now fix a $(p-1)$-tuple $\lambda_2,\ldots,\lambda_p \in \Lambda$, and define the sets
\[
B_1 = B \cap m_1\lambda_\eps.B \qand B_k = (S\lambda_\eps + n\lambda_k).B, \quad \textrm{for $k=2,\ldots,p$}.
\]
Note that by \eqref{bigint} and \eqref{nulowbnd}, combined with the fact that 
$\nu$ is $\Lambda(n)$-invariant, we have
\[
\nu(B_1) > \frac{\nu(B)^2}{2} \qand \nu(B_k) = \nu(n\lambda_k.(S\lambda_\eps).B) > 1 - \eps
\]
Hence,
\begin{align*}
\nu(B_1 \cap B_2 \cap \ldots \cap B_p)
&=
1 - \nu(B_1^c \cup B_2^c \cup \ldots B_p^2) \\[0.2cm]
&> 
1 - \sum_{k=1}^p \nu(B_k^c) = \nu(B_1) - (p-1)\eps \\[0.2cm]
&> \frac{\nu(B)^2}{2} - (p-1)\eps,
\end{align*}
and the right-hand side is strictly positive if $\eps$ is chosen so that
\[
\eps < \frac{\nu(B)^2}{2(p-1)}.
\]
If $1/3 \leq \nu(B)$ or $\mu(B) < 3 \cdot \nu(B),$ then
\[
\frac{1}{18(p-1)} \leq \frac{\nu(B)^2}{2(p-1)} \quad \textrm{or} \quad \frac{\mu(B)^2}{18(p-1)} < \frac{\nu(B)^2}{2(p-1)}
\]
respectively. Hence, if we take $\eps_o < \eps < \mu(B)^2/18(p-1)$, then, in either case,
\[
\nu(B_1 \cap B_2 \cap \ldots \cap B_p) > 0.
\]
After unwrapping the definitions of the sets $B_1,\ldots,B_p$, we conclude that with these choices of $\eps_o$ and $\eps$, there exist $m_2,\ldots,m_p \in S$ such that
\[
\nu(B \cap m_1\lambda_\eps.B \cap (m_2 \lambda_\eps + n\lambda_2).B \cap \ldots \cap (m_p \lambda_\eps + n\lambda_p).B) > 0.
\]
By \eqref{lowbnd2}, applied to the set
\[
F = \{0,m_1 \lambda_\eps,m_2\lambda_\eps + n\lambda_2,\ldots,m_p \lambda_\eps + n\lambda_p\} \subset \Lambda(n),
\]
we see that
\[
\mu(B \cap m_1\lambda_\eps.B \cap (m_2 \lambda_\eps + n\lambda_2).B \cap \ldots \cap (m_p \lambda_\eps + n\lambda_p).B) > 0,
\]
which finishes the proof.
\end{proof}

\section{Proof of Theorem \ref{ThmIntoComb}}
\label{sec:PrfComb}

Let $\Lambda$ be a free abelian group of rank $r$. If $E \subset \Lambda$ and 
$\cF = (F_N)$ is a F\o lner sequence in $\Lambda$, we define the \emph{upper density $\overline{d}_{\cF}(E)$ along $\cF$} by
\[
\overline{d}_{\cF}(E) = \varlimsup_{N \ra \infty} \frac{|E \cap F_N|}{|F_N|}.
\]
If $\Lambda = \bZ^r$ and $\cF_N = [-N,N]^d$, we see that $\overline{d}_{\cF}$ coincides with the upper density introduced in the introduction. The \emph{upper Banach density $d_\Lambda^*(E)$} is given by
\[
d^*_{\Lambda}(E) = \sup\{ \overline{d}_{\cF}(E) \, : \, \textrm{$\cF$ is a F\o lner sequence in $\Lambda$}\}.
\]
In this section we prove Theorem \ref{ThmIntoComb} in the following form.

\begin{theorem}
\label{ThmMainComb}
Let $E \subset \Lambda$ such that $d^*_{\Lambda}(E) > 0$. Then there is a positive integer $n$ such that
for all $p \geq 2$, there exist $\lambda \in \cP_\Lambda$ and $m_1 \in \bZ$ with the property that for all $\lambda_2,\ldots,\lambda_p \in \Lambda$, there are $m_2,\ldots,m_p \in \bZ \setminus \{0\}$ such that 
\[
d^*_\Lambda(E \cap (E-m_1 n\lambda) \cap (E-(m_2 n\lambda + n\lambda_2)) \cap \ldots \cap (E-(m_p n\lambda + n\lambda_p))) > 0.
\]
In particular, there is an element $\lambda_o \in E$ such that
\[
\lambda_o + m_1 n \lambda \in E, \enskip \lambda_o + m_2 n \lambda + n \lambda_2 \in E, \ldots, \lambda_o + m_p n \lambda + n \lambda_p \in E.
\]
\end{theorem}

\subsection{Proof of Theorem \ref{ThmMainComb}}

Let $E \subset \Lambda$ such that $d^*_\Lambda(E) > 0$. By the classical Furstenberg's Correspondence Principle (see for instance \cite[Proposition A.4]{BF}), we can find an ergodic $\Lambda$-space $(X,\mu)$
and a $\mu$-measurable set $B \subset X$ such that $d^*_\Lambda(E) = \mu(B) > 0$ and 
\begin{equation}
\label{corrprinciple}
d_\Lambda^*\Big(\bigcap_{\lambda \in F} \big(E-\lambda \big) \Big)
\geq 
\mu\Big(\bigcap_{\lambda \in F} \lambda.B  \Big),
\end{equation}
for every finite subset $F \subset \Lambda$. By Theorem \ref{ThmMainIntersections}, 
for every $p \geq 2$, we can find a positive integer $n$, a primitive element $\lambda \in \Lambda$ and $m_1 \in \bZ \setminus \{0\}$ with the property that for every $(p-1)$-tuple 
$\lambda_2,\ldots,\lambda_p \in \Lambda$, there are $m_1,\ldots,m_p \in \bZ \setminus \{0\}$ such that
\[
\mu\Big(\bigcap_{\lambda \in F} \lambda.B  \Big) > 0,
\]
where
\[
F = \{0,m_1n\lambda,m_2n\lambda+n\lambda_2,\ldots,m_pn\lambda + n\lambda_p\}.
\]
By \eqref{corrprinciple}, we conclude that
\[
d^*_\Lambda(E \cap (E-m_1 n\lambda) \cap (E-(m_2 n\lambda + n\lambda_2)) \cap \ldots \cap (E-(m_p n\lambda + n\lambda_p))) > 0.
\]

\end{document}